\documentclass{amsart}
\def\Hom{\operatorname{Hom}}
\newtheorem{Theo}{Theorem}
\newtheorem{Lem}[Theo]{Lemma}
\newtheorem{Cor}[Theo]{Corollary}
\begin{document}
\title[Homomorphisms into wreath products]{Homomorphisms from a finite
group into wreath products}
\author[J.-C. Schlage-Puchta]{Jan-Christoph Schlage-Puchta}
\address{Krijgslaan 281\\ Gebouw S22\\9000 Gent\\Belgium}
\email{jcsp@cage.ugent.be}
\subjclass{20P05, 20E22}
\keywords{Wreath products, Homomorphism numbers, Weyl groups}
\begin{abstract}
Let $G$ be a finite group, $A$ a finite abelian group. Each homomorphism
$\varphi:G\rightarrow A\wr S_n$ induces a homomorphism
$\overline{\varphi}:G\rightarrow A$ in a natural way. We show that as
$\varphi$ is 
chosen randomly, then the distribution of $\overline{\varphi}$ is
close to uniform. As application we prove a conjecture of T. M\"uller
on the number of homomorphisms from a finite group into Weyl groups of
type $D_n$.
\end{abstract}
\maketitle
Let $G$ be a finite group, $A$ a finite abelian group. In this article
we consider the number of homomorphisms $G\rightarrow A\wr S_n$, where
$n$ tends to infinity. These numbers are of interest, since they
encode information on the isomorphism types of subgroups of index $n$,
confer \cite{Iso1}, \cite{Iso2}. If $\varphi:G\rightarrow A\wr S_n$ is
a homomorphism, we can construct a homomorphism
$\overline{\varphi}:G\rightarrow A$ as follows. We represent the element
$\varphi(g)\in A\wr S_n$ as $(\sigma; a_1, \ldots, a_n)$, where
$\sigma\in S_n$ and $a_i\in A$, and then define
$\overline{\varphi}(g)=\prod_{i=1}^n a_i$. The fact that
$\overline{\varphi}$ is a homomorphism follows from the fact that $A$
is abelian and the definition of the product within a wreath
product. In this article we prove the following.
\begin{Theo}
\label{thm:main}
Let $G$ be a finite group of order $d$, $A$ a finite abelian group. Define the
distribution function $\delta_n$ on $\Hom(G, A)$ as the image of the
uniform distribution on $\Hom(G, A\wr S_n)$ under the map
$\varphi\mapsto\overline{\varphi}$. Then there exist positive
constants $c, C,$ independent of $n$, such that
$\|\delta_n-u\|_\infty<C e^{-cn^{1/d}}$, where $u$ is the
uniform distribution, and $\|\cdot\|_\infty$ denotes the supremum norm.
\end{Theo}
As an application we prove the following, which confirms a conjecture
by T. M\"uller.
\begin{Cor}
For a finite group $G$ there exists a constant $c>0$, such that if
$W_n$ denotes the Weyl group of type $D_n$, then
\[
|\Hom(G, W_n)| =
\left(\frac{1}{1+s_2(G)}+\mathcal{O}(e^{-cn^{1/d}})\right)|\Hom(G, C_2\wr S_n)|
\]
\end{Cor}
This assertion was proven by T. M\"uller under the assumption that $G$
is cyclic (confer \cite[Proposition~3]{wreath}). Different from his
approach we do not enumerate homomorphisms $\varphi$ with given image
$\overline{\varphi}$, but directly work with the distribution of
$\overline{\varphi}$, that is, we obtain the relation between
$|\Hom(G, W_n)|$ and $|\Hom(G, C_2\wr S_n)|$ without actually
computing these functions. 

Denote by $\pi:A\wr S_n\rightarrow S_n$ the canonical projection onto
the active group. The idea of the proof is to stratisfy the set
$|\Hom(G, A\wr S_n)|$ according to
$\pi\circ\varphi\in\Hom(G, S_n)$. It turns out that in strata such that
$\pi\circ\varphi(G)$ viewed as a permutation group on $\{1, \ldots,
n\}$ has a fixed point the distribution of
$\overline{\varphi}$ is actually uniform, while the probability of
having no fixed point is very small. 

\begin{Lem}
\label{Lem:uniform}
Let $\sigma:G\rightarrow S_n$ be a homomorphism such that
$\sigma(G)$ has a fixed point. Define the set
\[
M = \{\varphi:G\rightarrow A\wr S_n: \pi\circ\varphi=\sigma\}.
\]
Then the function $M\rightarrow\Hom(G, A)$ mapping $\varphi$ to
$\overline{\varphi}$ is surjective, and all fibres have the same
cardinality.
\end{Lem}
\begin{proof}
Without loss we may assume that the point $n$ is fixed. Let $\sigma_1$
be the restriction of $\sigma$ to the set $\{1, \ldots, n-1\}$. Then 
\[
M = \{\varphi_1:G\rightarrow A\wr S_{n-1} :
\pi\circ\sigma=\sigma_1\}\times \Hom(G, A),
\]
hence, for each $\psi:G\rightarrow A$ and each $\varphi_1:G\rightarrow
A\wr S_{n-1}$ with $\pi\circ\varphi_1=\sigma_1$ there is precisely one
$\varphi\in M$ with $\overline{\varphi}=\psi$ which coincides with
$\varphi_1$ on $A\wr S_{n-1}$. This implies that all fibres have the
same cardinality. Defining $\varphi:G\rightarrow A\wr S_n$ by
$\varphi(g)=(\sigma(g), 1, \ldots, 1)$ we see that $M$ is non-empty,
which implies the surjectivity.
\end{proof}

To bound the number of homomorphisms $\varphi$ for which
$\pi\circ\varphi$ has no fixed point we need the following, which is
contained in \cite[Proposition~1]{Iso1}, in particular the equality of
equations (8) and (9) in that article.
\begin{Lem}
Let $G$ be a group, $A$ a finite abelian group, $U\leq G$ a subgroup of
index $k$, $\varphi_1:G\rightarrow S_k$ the permutation representation
given by the action of $G$ on $G/U$. Then the number of homomorphisms
$\varphi:G\rightarrow A\wr S_k$ with $\pi\circ\varphi=\varphi_1$
equals $|A|^{k-1}|\Hom(U, A)|$.
\end{Lem}

We use this to prove the following.
\begin{Lem}
\label{Lem:Fix}
Let $G$  be a group of order $d$, $A$ a finite abelian group,
$\varphi:G\rightarrow A\wr S_n$ be a homomorphism chosen at 
random with respect to the uniform distribution. Then there is a
constant $c>0$, depending only on $G$, such 
that the probability that $\pi\circ\varphi(G)$ has no fixed points is
$\mathcal{O}(e^{-cn^{1/d}})$.
\end{Lem}
\begin{proof}
Let $U_1, \ldots, U_\ell$ be a complete list of subgroups of $G$ up to
conjugation, where $U_\ell = G$. To
determine a homomorphism $\varphi:G\rightarrow A\wr S_n$ we first have
to choose a homomorphism $\sigma:G\rightarrow S_n$, and then count the number
of ways in which this homomorphism can be extended to a homomorphism
into $A\wr S_n$. Suppose that the action of $G$ on $\{1, \ldots, n\}$
induced by $\sigma$ has $m_i$ orbits on which $G$ acts similar to
the action of $G$ on $G/U_i$. Then by the previous lemma we find that
there are
\[
\prod_{i=1}^\ell \big(|A|^{(G:U_i)-1}|\Hom(U_i, A)|\big)^{m_i}
\]
possibilities to extend $\sigma$. Next we compute the number of ways 
$\sigma$ can be chosen such that $\sigma$ realizes given values $m_1,
\ldots, m_\ell$. Choices of $\sigma$ correspond 
to subgroups of $S_n$ conjugate to some fixed subgroup with the given
number of orbits of the respective types, and the number of such
subgroups is $(S_n:C_{S_n}(\sigma(G))$. We have $C_{S_n}(\sigma(G))=
\times_{i=1}^\ell C_{\mathrm{Sym}(G/U_i)}(G)\wr S_{m_i}$, hence,
defining $c_i=|C_{\mathrm{Sym}(G/U_i)}(G)|$ we find that $\sigma$ can
be chosen in $\frac{n!}{\prod_{i=1}^\ell m_i! c_i^{m_i}}$ different
ways. Combining these results we obtain
\[
\label{eq:HomNumber}
|\Hom(G, A\wr S_n)| = n!\underset{m_1+\dots+m_\ell=n}{\sum_{m_1,\ldots,m_\ell}}
\prod_{i=1}^\ell\frac{\left(|A|^{(G:U_i)-1}|\Hom(U_i, A)|\right)^{m_i}} 
{m_i!c_i^{m_i}}.
\]
We claim that terms with $m_\ell=0$ are small when compared to the
whole sum. Since the number of summands is polynomial in $n$, it
suffices to show that for every tuple $(m_1, \ldots, m_{\ell-1}, 0)$
there exists a tuple $(m_1', \ldots, m'_{\ell-1}, m'_\ell)$ with
$m_\ell'\neq 0$, such that the summand corresponding to the first
tuple is smaller than the one corresponding to the second by a factor
$e^{cn^{1/d}}$. We do so by explicitly constructing the second
tuple. Without loss we may assume that in the first
tuple $m_1$ is maximal. We then set $m_1'=m_1-\lfloor cn^{1/d}\rfloor$,
$m_\ell'=(G:U_1)\lfloor cn^{1/d}\rfloor$, and $m_i'=m_i$ for $i\neq 1,
\ell$, where $c$ is a positive constant chosen later. Then the
product on the right hand side of the last displayed equation changes by a
factor
\[
\frac{m_1!}{(m_1-\lfloor cn^{1/d}\rfloor)!}\left(\frac{|A|^{(G:U_1)-1}|
\Hom(U_1, A)|}{c_1|\Hom(G, A)|^{(G:U_1)}}\right)^{-\lfloor cn^{1/d}\rfloor}
\frac{1}{\big((G:U_1)\lfloor cn^{1/d}\rfloor\big)!}.
\]
We may assume that $n$ is sufficiently large, so that $m_1>2\lfloor
cn^{1/d}\rfloor$. We can then estimate the factorials using the
largest and smallest factors occurring. The other terms can be bounded
rather careless to find that this quotient is
at least
\[
\left(\frac{m_1}
{\big(cdn^{1/d}|A|\big)^{d}|\Hom(U_1, A)|}\right)^{\lfloor cn^{1/d}\rfloor}.
\]
Since $m_1$ was chosen maximal we have $m_1\geq\frac{n}{|G|\ell}$, and
we find that for $c^{-1}= ed\ell|A||\Hom(U_1, A)|$ the last
expression is at least $e^{\lfloor cn^{1/d}\rfloor}$. Since $c$
depends only on the subgroup $U_1$, we can take the minimum value over
all the finitely many subgroups and obtain that there exists an
absolute constant $c>0$, such that the number of homomorphisms
$\varphi$ such that $\pi\circ\varphi$ has no fixed point is smaller by
a factor $\mathcal{O}(e^{-cn^{1/d}})$ than the number of all
homomorphisms.
\end{proof}

To prove the theorem let $\varphi:G\rightarrow A\wr S_n$
be chosen with respect to the uniform distribution. Let $p$ be the
probability that $(\pi\circ\varphi)(G)$ has no fixed point. By
Lemma~\ref{Lem:uniform} we see that the conditional distribution of
$\overline{\varphi}$ subject to the condition that
$(\pi\circ\varphi)(G)$ has a fixed point is uniform, hence 
$\delta=(1-p)u+p\delta^0$ for some distribution function
$\delta^0$. This implies $\|\delta-u\|_\infty\leq p$. By
Lemma~\ref{Lem:Fix} we see that $p=\mathcal{O}(e^{-cn^{1/d}})$, and
our claim follows.

To deduce the corollary note that $W_n$ is the subgroup of $C_2\wr
S_n$ defined by the condition $(\pi; a_1, \ldots, a_n)\in 
W_n\Leftrightarrow a_1\cdots a_n=1$, that is, a homomorphism
$\varphi:G\rightarrow C_2\wr S_n$ has image in $W_n$ if and only if
$\overline{\varphi}:G\rightarrow C_2$ is trivial. By the theorem the
probability for this event differs from the probability that a random
homomorphism $G\rightarrow C_2$ is trivial by
$\mathcal{O}(e^{-cn^{1/d}})$, hence, we have
\[
|\Hom(G, W_n)| = \left(\frac{1}{|\Hom(G, C_2)|}+\mathcal{O}(e^{-cn^{1/d}})\right)
|\Hom(G, C_2\wr S_n)|.
\]
But there is a bijection between non-trivial homomorphisms
$G\rightarrow C_2$ and subgroups of index 2, hence, $|\Hom(G,
C_2)|=1+s_2(G)$, and the corollary follows.

\end{document}